\documentclass{amsart}

\usepackage{amssymb,mathtools,tikz}
\usepackage[alphabetic,initials,nobysame]{amsrefs}
\usepackage[shortlabels]{enumitem}
\usepackage{hyperref}

\BibSpec{article}{%
    +{}  {\PrintAuthors}                {author}
    +{,} { \textit}                     {title}
    +{.} { }                            {part}
    +{:} { \textit}                     {subtitle}
    +{,} { \PrintContributions}         {contribution}
    +{.} { \PrintPartials}              {partial}
    +{,} { }                            {journal}
    +{}  { \textbf}                     {volume}
    +{}  { \PrintDatePV}                {date}
    +{,} { \issuetext}                  {number}
    +{,} { \eprintpages}                {pages}
    +{,} { }                            {status}
    +{,} { \url}                        {url}
    +{,} { \PrintDOI}                   {doi}
    +{,} { available at \eprint}        {eprint}
    +{}  { \parenthesize}               {language}
    +{}  { \PrintTranslation}           {translation}
    +{;} { \PrintReprint}               {reprint}
    +{.} { }                            {note}
    +{.} {}                             {transition}
    +{}  {\SentenceSpace \PrintReviews} {review}
}

\newcommand{\br}{\overline}
\newcommand{\R}{\mathbb R}

\newcommand{\N}{\mathbb N}

\theoremstyle{plain}
\newtheorem{theorem}{Theorem}
\newtheorem{lemma}[theorem]{Lemma}

\newtheorem{prop}[theorem]{Proposition}
\newtheorem{corollary}[theorem]{Corollary}

\theoremstyle{definition}
\newtheorem{definition}[theorem]{Definition}

\theoremstyle{remark}

\newtheorem{question}{Question}

\DeclareMathOperator{\diam}{\textup{\text{diam}}}

\DeclareMathOperator{\loc}{\textup{loc}}

\numberwithin{equation}{section}
\numberwithin{theorem}{section}

\renewcommand{\S}{\mathbb{S}}

\begin{document}
\title{Non-removability of Sierpi\'nski spaces}

\author{Dimitrios Ntalampekos}
\address{Institute for Mathematical Sciences, Stony Brook University, Stony Brook, NY 11794, USA.}
\email{dimitrios.ntalampekos@stonybrook.edu}

\author{Jang-Mei Wu}
\address{Department of Mathematics, University of Illinois, Urbana, IL 61822, USA.}
\email{jmwu@illinois.edu}

\date{\today}

\begin{abstract}
We prove that all Sierpi\'nski spaces in ${\S}^n$, $n\geq 2$, are non-removable for (quasi)conformal maps, generalizing the result of the first named author \cite{Ntalampekos:CarpetsNonremovable}. More precisely, we show that for any Sierpi\'nski space $X\subset \S^n$ there exists a homeomorphism $f\colon \S^n\to \S^n$, conformal in $\S^n\setminus X$,  that  maps $X$ to a set of positive measure and is not globally (quasi)conformal. This is the first class of examples of non-removable sets in higher dimensions.
\end{abstract}
\maketitle

\section{Introduction}

In this work we approach the problem of (quasi)conformal removability in $\S^n$ for $n\geq 3$. A compact set $K\subset \S^n$ is \textit{(quasi)conformally removable} if every homeomorphism $f\colon \S^n\to \S^n$ that is (quasi)conformal in $\S^n\setminus K$ is, in fact, (quasi)conformal everywhere. In dimension $2$, a set is conformally removable if and only if it is quasiconformally removable; see \cite{Younsi:removablesurvey} for a survey of results. The unavailability of the techniques involving the Beltrami equation in higher dimensions does not allow us to draw such a conclusion, and so far we only have the trivial implication that quasiconformal removability implies conformal removability for sets of measure zero, because $1$-quasiconformal mappings are conformal \cite[Theorem 15]{Gehring:1QC}.

Note that there are sets of positive measure that are conformally removable, a phenomenon that does not occur in dimension $2$. For example, let $K\subset \S^n$, $n\geq 3$, be a compact set with empty interior and positive measure such that $\S^n \setminus K$ is connected. Then a homeomorphism $f\colon \S^n\to \S^n$ that is conformal on $\S^n\setminus K$ is actually a M\"obius map on $\S^n\setminus K$ (by Liouville's Theorem \cite[Section 29]{Gehring:1QC}), and thus on $\S^n$ by continuity. This implies that $K$ is conformally removable. On the other hand, if $K=C\times [0,1]^{n-1}$, where $C\subset \R$ is a Cantor set (with positive measure or measure zero), then one can show that $K$ is non-removable for quasiconformal maps; see for example the argument in \cite[p.\ 6]{Ntalampekos:gasket}. Since $\R^n \setminus K$ is connected, it follows as before that $K$ is conformally removable and thus conformal and quasiconformal removability are not equivalent in dimensions greater than $2$.

The techniques used to prove that a set is (quasi)conformally removable in higher dimensions are the same as the planar ones. In particular, all sets of $\sigma$-finite Hausdorff $(n-1)$-measure are removable in $\S^n$ \cite[Theorem 35.1]{Vaisala:quasiconformal}, as also are boundaries of domains satisfying a certain quasihyperbolic condition \cite{JonesSmirnov:removability}.

There are very few non-trivial examples of quasiconformally non-removable sets in higher dimensions (e.g.\ \cite{Bishop:NonremovableR3}), the main difficulty being the lack of tools for the construction of homeomorphisms with good control of the quasiconformal distortion; in contrast, for planar constructions see \cite{Bishop:flexiblecurves}, \cite{KaufmanWu:removable},  \cite{Kaufman:graphnonremovable}, \cite{Ntalampekos:gasket}, \cite{Ntalampekos:CarpetsNonremovable}. It is not even known whether all sets of positive measure are quasiconformally non-removable. We pose a stronger question (see also \cite[Question 3]{Bishop:flexiblecurves}), known to have a positive answer (partially) in $2$ dimensions \cite{KaufmanWu:removable}:
	\begin{question}
Let $K\subset \S^n$ ($n\geq 2$) be a compact set of positive Lebesgue measure. Does there exist a homeomorphism $f$ of $\S^n$ that is quasiconformal on $\S^n\setminus K$ and maps $K$ (or a subset of $K$ of positive measure) to a set of measure zero?
	\end{question}

Note that a positive answer to this question would imply that all sets of positive Lebesgue measure are non-removable for quasiconformal maps.

In this work we prove that a large class of sets, namely Sierpi\'nski spaces, are non-removable for (quasi)conformal maps in $\S^n$. This generalizes the $2$-dimensional result from work of the first named author \cite{Ntalampekos:CarpetsNonremovable}. Sierpi\'nski spaces are higher-dimensional analogs of planar Sierpi\'nski carpets.

\begin{definition}\label{def:Sierpinski_space}
A continuum $X\subset \S^n$, $n\geq 2,$ is an $(n-1)$-dimensional Sierpi\'nski space if its complement $\S^n\setminus X$ consists of countably many components $U_i$, $i\in \N$, satisfying the following conditions:
\begin{enumerate}[\textup(1)]
\item $\S^n\setminus U_i$ is an $n$-cell for each $i\in \N$,
\item $\br {U_i}\cap \br {U_j}=\emptyset$ for $i\neq j$,
\item $\bigcup_{i\in \N} U_i$ is dense in $\S^n$, and
\item $\diam(U_i)\to 0$ as $i\to \infty$.
\end{enumerate}
\end{definition}

In $\S^n$, $n\geq 3$,
the boundary components $\partial U_i$, $n\in\N$, of $\S^n\setminus X$ are not assumed to be flat spheres in Definition \ref{def:Sierpinski_space}.
In  $\S^2$,  condition $(1)$  is equivalent to requiring that $\partial U_i$, $i\in \N$, are Jordan curves, or equivalently flat $1$-spheres.

All $(n-1)$-Sierpi\'nski spaces in $\S^n$, for a fixed $n\geq 2$, are homeomorphically equivalent. This topological result was proved by Whyburn for $n=2$, and by  Cannon for dimensions $n=3$ and $n\geq 5$. Cannon's proof  is based on the known Annulus Theorem at that time. Since the Annulus Conjecture has now been proved in dimension $4$ (see Theorem \ref{Annulus:thm}), Cannon's proof extends to  $n=4$ also.

\begin{theorem}[\cite{Whyburn:theorem}, \cite{Cannon:Sierpinski}]\label{Whyburn_Cannon:thm}
If $X,Y\subset \S^n$, $n\geq 2$, are $(n-1)$-dimensional Sierpi\'nski spaces, then there exists a homeomorphism $f\colon X\to Y$.
\end{theorem}

Our main result is the following:

\begin{theorem}\label{Theorem:main}
Let $X\subset \S^n$, $n\geq 2$, be an $(n-1$)-dimensional Sierpi\'nski space. Then there exist a Sierpi\'nski space $Y \subset \S^n$ of positive Lebesgue measure and a homeomorphism $f\colon \S^n\to \S^n$ which maps $X$ onto $Y$ and  is conformal on $\S^n\setminus X$.
\end{theorem}

The statement is  different from the $2$-dimensional result in \cite{Ntalampekos:CarpetsNonremovable}. It is proved in \cite{Ntalampekos:CarpetsNonremovable} that if $X,Y \subset \S^2$ are \textit{any}  Sierpi\'nski carpets (i.e., $1$-dimensional Sierpi\'nski spaces), then there exists a homeomorphism $f\colon \S^2\to \S^2$ which maps $X$ onto $Y$ and is conformal on $\S^2\setminus X$. We do not expect such a strong statement in higher dimensions.  Firstly, the boundary components $\partial {U_i}$, $i\in \mathbb N$, of the complement of a Sierpi\'nski space are not necessarily flat spheres. Secondly, even under the stronger assumption that all $\partial U_i$, $i\in \mathbb N$, are flat $(n-1)$-spheres in $\S^n$, topological open balls  $U_i$, $i\in \mathbb N$, need not be quasiconformally equivalent to an open Euclidean ball  \cite{GehringVaisala:QuasiconformalitySpace} -- in contrast to the $2$-dimensional case in which the Riemann Mapping Theorem can be invoked.

The  proof of Theorem \ref{Theorem:main} follows the lines of Whyburn and Cannon. However, in order to prove the (quasi)conformal non-removability of an $(n-1)$-Sierpi\'nski space $X$ in $\S^n$, we are not allowed to alter the topology of the complementary components of $X$ in  $\S^n$, but we can only use (quasi)conformal deformations of them. For this reason, we use a decomposition of $\S^n$ whose degenerate elements are not necessarily $n$-cells; see Lemma \ref{decomposition:lemma}. This entails some technical complications.

\begin{corollary}\label{Cor:nonremovability}
All $(n-1)$-dimensional Sierpi\'nski spaces in $\S^n$, $n\geq 2$, are non-removable for (quasi)\-conformal maps.
\end{corollary}

Since a (quasi)conformal map $f$ of $\S^n$ necessarily maps sets of measure zero to sets of measure zero \cite[Theorem 33.2]{Vaisala:quasiconformal}, this Corollary follows immediately from Theorem \ref{Theorem:main} for Sierpi\'nski spaces having zero measure. If a Sierpi\'nski space has positive measure, then the proof of the non-removability requires an extra ingredient and it is given at the end of Section \ref{section:mainproof}.

We give the proof of Theorem \ref{Theorem:main} in Section \ref{section:mainproof}, based on a topological lemma (Lemma \ref{lemma:subdivisions}) proved in Section \ref{section:subdivisions}.
Finally, Section \ref{section:facts} contains several topological facts that are used throughout the paper.

\subsection*{Acknowledgments:}
The authors would like to thank the anonymous referee for reading the paper carefully and providing insightful comments. Part of this research was conducted while the first named author was visiting University of Illinois Urbana-Champaign. He thanks the faculty and the staff of the Department of Mathematics for their hospitality. The research of the second named author is partially supported by  Simons Foundation Collaboration Grant $\# 353435$.

\section{Proof of main result}\label{section:mainproof}

An $n$-dimensional CW-complex $K$ is a \emph{cubical complex} if  each cell in $K$ is isomorphic to a unit cube $[0,1]^k$ for some $0\leq k \leq n$, and the intersection $\sigma \cap \sigma'$
of any $k$-cells $\sigma$ and $\sigma'$, if nonempty,  is a $j$-cell in $K$ for some $0\leq j \leq k-1$. The $k$-cells in $K$ are  called $k$-cubes, and the  subcomplex $K^{[k]}$ consisting of all cubes of dimension at most $k$ is called the \textit{$k$-skeleton of $K$}.
The union  $|K|$ of all cubes in $K$ is its space.

\begin{definition}
Let $\varepsilon >0$, and $Y$ be  a closed subset of $\S^n$.  A collection  $\{Y_{\alpha}\}$ of $n$-cubes  is an \emph{$\varepsilon$-subdivision of $Y$} if
there exists a  finite  cubical complex $K$ which  has $\{Y_{\alpha}\}$  as all its $n$-cubes, each of which has diameter less than $\varepsilon$, and has space  $|K|=\bigcup Y_{\alpha}=Y$.

Let $X \subset \S^n$ be an $(n-1)$-dimensional Sierpi\'nski space.
A collection  $\{X_{\alpha}\}$ of $(n-1)$-dimensional Sierpi\'nski spaces  is called
an \emph{$\varepsilon$-subdivision  of $X$},
if  there
exist  components $U_1,\dots,U_N$ of $\S^n\setminus X$  and
an $\varepsilon$-subdivision  $\{Y_{\alpha}\}$ of the set $Y \coloneqq\S^n \setminus \bigcup_{j=1}^N  U_j$, having the same number of elements as that of $\{X_{\alpha}\}$, for which
\begin{enumerate}
\item the boundaries of the $n$-cubes in $\{Y_{\alpha}\}$ do not intersect $\bigcup_{j=N+ 1}^\infty \overline{U_j}$, and

\item $X_{\alpha}= X\cap Y_{\alpha}$ for each $\alpha$.

\end{enumerate}
In this case, we call  cube $Y_{\alpha}$ the  \emph{hull} of the Sierpi\'nski space $X_{\alpha}$.

\end{definition}

\bigskip

We now state the key lemma for the proof of Theorem \ref{Theorem:main}.

\begin{lemma}\label{lemma:subdivisions}
Let $n\geq 2$, and $X$ be an $(n-1)$-dimensional Sierpi\'nski space in $\S^n$.
Let $\varepsilon >0$, and $U_1,\dots,U_N$ be a collection of components in $\S^n\setminus X$ for which the remaining components have diameters less than $\varepsilon$, and let  $G_1,\dots,G_{N} \subset \S^n$ be open sets with pairwise disjoint closures for which $\S^n\setminus G_1,\dots,\S^n \setminus G_{N}$ are $n$-cells. Given $N$
  orientation-preserving homeo\-morphisms $h_j\colon \partial U_j \to \partial G_j$, $j\in \{1,\dots,N\}$, there exist
 a homeomorphism
 \[h\colon \S^n\setminus \bigcup_{j=1}^{N} U_j \to \S^n\setminus \bigcup_{j=1}^{N} G_j \]
  which extends $h_j$, $j\in \{1,\dots,N\}$, an  $\varepsilon$-subdivision $\{X_\alpha\}$ of $X$,  and an $\varepsilon$-subdivision $\{Z_\alpha\}$ of $\S^n \setminus \bigcup_{j=1}^N G_j$ so that  $\{h^{-1}(Z_\alpha)\}$ are the hulls of the Sierpi\'nski spaces
 $\{X_\alpha\}$.

\end{lemma}

The proof of this lemma is given in the next section.

\begin{proof}[Proof of Theorem \ref{Theorem:main}]
Let $\varepsilon_k=1/k$,  and fix a component $U_{0}$ of $\S^n\setminus X$ that has the largest diameter.

First let  $\varepsilon = \varepsilon_1$, and $U_0, U_1,\ldots, U_N$ be  a collection of components in $\S^n\setminus X$ for which the  remaining components  have diameters less than $\varepsilon_1$.  Set $G_0\coloneqq U_0$, and let
$h_{1}$ be an embedding of  $\bigcup_{j=0}^N \overline{U_j}$ into $\S^n$, which is the identity map on $\br{U_0}$, and is a similarity on $\br {U_j}$ that shrinks and translates ${U_j}$  to a set  $G_j$ for each $j\in \{1,\dots,N\}$, so that the sets $\br{G_j}$ are disjoint subsets of $\S^n\setminus \br{U_0}$.
Then, by Lemma \ref{lemma:subdivisions}, there exist a global homeomorphic extension $h_1\colon \S^n\to \S^n$, an $\varepsilon_1$-subdivision $\{X_\alpha\}$ of $X$,  and an $\varepsilon_1$-subdivision $\{Z_\alpha\}$ of $Z\coloneqq \S^n \setminus \bigcup_{j=0}^N G_j$  so that cubes in $\{h_1^{-1}(Z_\alpha)\}$ are the hulls of the Sierpi\'nski spaces in
 $\{X_\alpha\}$, respectively.

In the second step, let $X_{\alpha_0}$ be a  Sierpi\'nski space in the $\varepsilon_1$-subdivision of $X$, and $Z_{\alpha_0}$ be the corresponding $n$-cell in the $\varepsilon_1$-subdivision of $Z$ under $h_1^{-1}$. Let $V_0$ be the complementary component of $X_{\alpha_0}$ that contains $U_0$, and fix a finite collection of components
 $V_0, V_1,\dots,V_{M}$  in $\S^n\setminus X_{\alpha_0}$ for which
  all remaining components have diameters less than
$\varepsilon_2$. We observe that $\S^n\setminus  V_0= h_1^{-1}(Z_{\alpha_0}) $.
Again, by Lemma \ref{lemma:subdivisions}, there exist  a homeomorphism $h_2$ from  $\S^n\setminus V_0$ onto $Z_{\alpha_0}$, and $\varepsilon_2$-subdivisions of   $X_{\alpha_0}$ and  $Z_{\alpha_0}$, respectively,  for which
\begin{enumerate}
\item $h_2$ agrees with
$h_1$ on $\partial V_0$, and shrinks and translates  $V_1, \ldots, V_M$ to sets  $D_1,\ldots, D_M$, having pairwise disjoint closures,  in the interior of $Z_{\alpha_0}$ by similarities,  and
\item  the cubes in the preimage, under $h_2$, of the $\varepsilon_2$-subdivision of $Z_{\alpha_0}$ are the hulls of the Sier\-pi\'nski  spaces in the $\varepsilon_2$-subdivision of $X_{\alpha_0}$.
\end{enumerate}
We repeat this for each Sierpi\'nski space $X_\alpha$ in the $\varepsilon_1$-subdivision  of $X$, and extend $h_2$ to a homeomorphism $\S^n \to \S^n$ which agrees with $h_1$ on $\bigcup_{j=0}^N \br{U_j}$.

Inductively, we obtain a sequence of homeomorphisms $h_k\colon \S^n\to \S^n$ for which $h_k$ is $\varepsilon_k$-close to $h_m$ uniformly for all $m\geq k$. The same statement holds for the inverses $h_k^{-1}$. In view of  the inductive construction, on each component of $\S^n\setminus X$, the sequence $\{h_k\}_{k\in \N}$ is eventually constant and, in fact, the maps are eventually conformal. Indeed, if $U$ is a component of $\S^n\setminus X$ with $\varepsilon_k\leq \diam(U)$, then $h_{k}|_U$ is a similarity and $h_m|_U= h_{k}|_U$ for all $m\geq k$.

We may conclude that  the sequence $h_k$ converges uniformly to a homeomorphism $h\colon \S^n\to \S^n$ which has  the following properties:
\begin{enumerate}[(a)]
\item $h$ is conformal in the complement of $X$,
\item $Y\coloneqq h(X)$ is a Sierpi\'nski space, and
\item $Y$ has positive $n$-measure.
\end{enumerate}
For the latter property one has to note that the image of $\S^n\setminus X$ under $h$ has smaller Lebesgue measure, since $h$ is a similarity on each component of $\S^n\setminus X$ and it shrinks all but one component. In fact, $h$ may be chosen so that the measure of  $Y$ is arbitrarily close to that  of $\S^n$.
 \end{proof}

\begin{proof}[Proof of Corollary \ref{Cor:nonremovability}]
Identify the Sier\-pi\'nski space $X$ with a compact subset of $\R^n\subset \R^n\cup\{\infty\}\approx \S^n$ by the stereographic projection, and let $U_0$ be the unbounded component of $\R^n\setminus X$.
Enumerate the remaining  complementary components of $X$ by $U_i$, $i\in \N$.

By the preceding proof, we may obtain a homeomorphism $h$ of $\R^n$ which is the identity on  $U_0$ and is,  for each $i\in \N$, a similarity on  $U_i$ that shrinks $U_i$ by any desired factor.
In particular, we may require that $|h(U_i)|\leq |U_i|^{i+1}$ for $i\in \N$, where $|\cdot |$ denotes Lebesgue measure in $\R^n$. Note that the Jacobian of $h$ on $U_i$ is a constant equal to $|h(U_i)|/|U_i|$.

If $h^{-1}$ were quasiconformal, then its Jacobian $J_{h^{-1}}$ would be in  $L^{1+\varepsilon}_{\loc}(\R^n)$ for some $\varepsilon>0$ depending on $h^{-1}$; see \cite[Theorem 1]{Gehring:Lpintegrability}.
On the other hand,

\begin{align*}
\int_{\bigcup_{i\in \N}h(U_i)} (J_{h^{-1}})^{1+\varepsilon}&= \sum_{i\in \N} \frac{|U_i|^{1+\varepsilon}}{|h(U_i)|^{1+\varepsilon}}|h(U_i)| \\
&= \sum_{i\in \N} |U_i|^{1+\varepsilon} |h(U_i)|^{-\varepsilon} \geq \sum_{i\in \N} |U_i|^{1-\varepsilon i},
\end{align*}
which diverges because $\diam(U_i)\to 0$. This contradiction proves that $h^{-1}$, and thus $h$, is not quasiconformal on $\R^n$.

Since $h|_{\R^n\setminus X}$ is (quasi)conformal in the complement  of $X$ and $h$ is its unique homeomorphic extension to $\R^n$, we conclude that the set $X$ is non-removable.
\end{proof}

\section{Proof of Lemma \ref{lemma:subdivisions}}\label{section:subdivisions}

Under the assumptions of Lemma \ref{lemma:subdivisions}, we first prove a decomposition result suitable for our setting (compare to the statement of Theorem \ref{Decomposition}):

\begin{lemma}\label{decomposition:lemma}
Under the assumptions of Lemma \ref{lemma:subdivisions}, there exists a continuous surjective map $p\colon \S^n\to \S^n$ which fixes $\bigcup_{i=1}^N \br{U_j}$, and induces a decomposition of $\S^n$ into  sets $\{ \br{U_j}\}_{j\geq N+1}$ and points. Specifically,  there exist countably many points $\{q_j\colon j\geq N+1\}$ for which $p^{-1}(q_j)=\br{U_j}$,  and the map $p \colon \S^n\setminus \bigcup_{j=N+1}^{\infty} \br{U_j}\to \S^n\setminus \{q_j\colon j\geq N+1\}$ is bijective.
\end{lemma}
\begin{proof}
Consider first an embedding $f\colon X \to \S^n$ such that  boundary components of the complement of $X'\coloneqq f(X)$ in $\S^n$ are flat $(n-1)$-spheres. Such a map exists by Lemma \ref{Cannon:lemma}. We remark that each sphere $f(\partial U_j)$ bounds a complementary component, denoted by $U_j'$, of $X'$, and that $X'=\S^n\setminus \bigcup_{j=1}^\infty U'_j$. For an explanation of this remark, see for example the argument in \cite[Lemma 5.5]{Bonk:uniformization}.

Consider the homeomorphisms $g_j= (f|_{\partial U_j})^{-1} \colon \partial U_j'\to \partial U_j$, $j=1,\dots,N$. By Corollary \ref{Extension:cor}, these maps may be extended to a homeomorphism $g\colon \S^n\setminus \bigcup_{j=1}^N U_j'\to \S^n\setminus \bigcup_{j=1}^N U_j$. Note that $g$ maps each $\br{U_j'}$ to a flat $n$-cell, for $j\geq N+1$.

Consider the homeomorphism $G=g\circ f$ on $X$. Since $G$  is the identity on $\bigcup_{j=1}^N \partial U_j$,  it may be extended to be the identity map on $\br \Omega$, where $\Omega\coloneqq \bigcup_{j=1}^N {U_j}$. Moreover, for $j\geq N+1$, $G$ maps $\partial U_j$ to a flat sphere that bounds a complementary component, denoted by $U_j''$, of $X''=G(X)$.

We apply the Decomposition Theorem \ref{Decomposition} to obtain a map $p''\colon \S^n\to \S^n$ that fixes $\br \Omega$, and collapses $\br {U_j''}$, $j\geq N+1$, to points. The  composition
\[p\coloneqq p''\circ G\colon X\cup \Omega \to \S^n\]
  is the identity on $\br\Omega$, and  collapses  $\partial U_j$ to a point $q_j$ satisfying $p^{-1}(q_j)=\partial U_j$ for  $j\geq N+1$. We now extend $p$ to $\bigcup_{j=N+1}^\infty {U_j}$ so that $p^{-1}(q_j)=\br{U_j}$. Since $\diam(U_j)\to 0$, the extension $p \colon \S^n \to\S^n$ is continuous, and hence is the map claimed in the lemma.
\end{proof}

\begin{proof}[Proof of Lemma \ref{lemma:subdivisions}]
Let $p\colon \S^n \to\S^n$ be the map in Lemma \ref{decomposition:lemma}. By  Corollary \ref{Extension:cor}, the homeomorphisms
\[h_j\circ (p|_{\partial{U_j}})^{-1}\colon p(\partial{U_j}) \to \partial{G_j}, \quad 1\leq j\leq N ,  \]
on  subsets of $ \S^n$, can be extended to a homeomorphism
\[H\colon p(\S^n) \setminus \bigcup_{j=1}^{N} p(U_j) \to \S^n\setminus \bigcup_{j=1}^{N} G_j.\]
Fix a number $\delta \in (0,\varepsilon)$, for which all sets  of diameter $\delta$ are mapped by   $p^{-1}\circ H^{-1}$  to sets of diameters less than $\varepsilon$. The existence of such a $\delta$ is a consequence of the uniform continuity of $H^{-1},p$ and the fact that the preimages of points under $p$ have diameter smaller than $\varepsilon$; recall from the statement of Lemma \ref{lemma:subdivisions} that the sets $\br{U_j}$, $j\geq N+1$, that are collapsed to points have diameters smaller than $\varepsilon$.

Fix next a finite cubical complex $K$ on $\S^n\setminus \bigcup_{j=1}^N G_j$ whose $n$-cubes have diameters less than $\delta$, and whose  $(n-1)$-skeleton $K^{[n-1]}$ does not meet the countable set $A\coloneqq  \{H (q_j)\colon j\geq N+1\}$. The cubical complex $K$ may be found by identifying $\S^n$ with $\R^n \cup\{\infty\}$ and $\{G_j \colon 1\leq j\leq N\}$ with Euclidean cubes having edges parallel to coordinate axes, with the help of Corollary \ref{Extension:cor}. Then, the family of $n$-cells $\{C\colon C\,\,\text{is an}\,\,n\text{-cube in}\,\,K\}$ is a $\delta$-subdivision of $Z\coloneqq \S^n \setminus \bigcup_{j=1}^N G_j$.

Observe that $H\circ p\colon \S^n \setminus \bigcup_{j=1}^N U_j \to \S^n \setminus \bigcup_{j=1}^N G_j$ is a cellular map between two compact manifolds \emph{with boundary}.
For the purpose of applying the Approximation Theorem  (Corollary \ref{Approximation:cor}),
we attach, on the domain side, an $n$-cell $C_j$  to $\S^n \setminus \bigcup_{j=1}^N U_j $ along $\partial U_j$ for every $j=1,\ldots, N$, to obtain an expanded space $\widetilde{M}$ which is homeomorphic to $\S^n$, and we do the same on the target side to obtain an expanded space $\widetilde{N}$ homeomorphic to $\S^n$. We extend $H\circ p$ to a map $\widetilde{M}\to \widetilde{N}$, which is  a homeomorphism between every pair of added $n$-cells.
We now apply Corollary \ref{Approximation:cor} to conclude that
$\{(H\circ p)^{-1}(C)\colon C\,\,\text{is an}\,\,n\-\text{-cube in}\,\,K\}$ are $n$-cells, and that they form an $\varepsilon$-subdivision of $Y\coloneqq \S^n \setminus \bigcup_{j=1}^N U_j$.
Thus,
$\{X\cap (H\circ p)^{-1}(C)\colon C\,\,\text{is an}\,\,n\text{-cube in}\,\,K\}$ is an $\varepsilon$-subdivision of $X$.

Observe that $(H\circ p)^{-1}$ is injective on the space $|K^{[n-1]}|$ of the $(n-1)$-skeleton $K^{[n-1]}$, since the latter does not meet the set $A$. Set $L= (H\circ p)^{-1} (|K^{[n-1]}|)$.
The homeomorphism $h$ in Lemma \ref{lemma:subdivisions} may be obtained by first setting
 \[h|_L= H \circ p|_L,\]
and then extending $h|_L$ to the interior of the $n$-cells in the $\varepsilon$-subdivision of $Y$ homeomorphically. This completes the proof of Lemma \ref{lemma:subdivisions}.
\end{proof}

\medskip

\section{Topological Facts}\label{section:facts}

We record some topological facts that  are needed for the proof of Lemma \ref{lemma:subdivisions}. We state them  and provide references for $n\geq 3$, but all these statements are also true for $n=2$. We refer the reader to \cite{Daverman-Venema} for the definitions of the various topological notions appearing below.

\begin{theorem}[Decomposition Theorem, \cite{Moore:theorem}, \cite{Meyer_DV}, {\cite[II.8.6A]{Daverman:decompositions}}]\label{Decomposition}
Let $n\geq 2$, $\{B_i\}_{i\in \N}$ be a null sequence of disjoint flat $n$-cells in $\S^n$, and $U$ be an open set containing $\bigcup_{i\in \N} B_i$. Then there exists a continuous surjective map $f\colon \S^n\to \S^n$, which is the identity outside $U$, such that $f$ induces a decomposition of $\S^n$ into the sets $\{B_i\}_{i\in \N}$ and points. Specifically,  there exist countably many points $p_i\in U$, $i\in \N$, such that $f^{-1}(p_i)=B_i$ for each $i\in \N$, and the map $f\colon  \S^n\setminus \bigcup_{i\in \N} B_i \to \S^n\setminus \{p_i\colon i\in \N\}$ is bijective.
\end{theorem}

\begin{lemma}[{\cite[Lemma 0]{Cannon:Sierpinski}}]\label{Cannon:lemma}
Let $n\geq 3$, and $X$ be an $(n-1)$-dimensional Sierpi\'nski space in $\S^n$. Then there exists an embedding $h\colon X\to \S^n$ such that the boundary components of $\S^n\setminus h(X)$ are flat $(n-1)$-spheres.
\end{lemma}

\begin{theorem}[Annulus Theorem, \cite{Moise:Affine}, \cite{Kirby:Annulus}, \cite{Quinn:EndsIII}]\label{Annulus:thm}
Let $n\geq 3$, and $D_1,D_2\subset \S^n$ be disjoint flat $n$-cells. Then $\S^n\setminus (D_1\cup D_2)$ is homeomorphic to $\S^{n-1}\times[0,1]$.
\end{theorem}

The following is a consequence of  the Annulus Theorem; see Kirby \cite{Kirby:Annulus}.

\begin{theorem}[Isotopy Theorem]\label{Isotopy:thm}
Let $n\geq 3$, and $f\colon \S^n\to \S^n$ be an orientation-preserving homeomorphism. Then $f$ is isotopic to the identity.
\end{theorem}

The following extension property follows from the Annulus Theorem almost immediately. We state it here for completeness.

\begin{prop}[Extension]\label{Extension:prop}
Let $n\geq 3$,  $U_1,\dots,U_N\subset \S^n$ be open sets with pairwise disjoint closures and whose boundaries are flat $(n-1)$-spheres, and let $ U'_1,\dots,U'_N\subset \S^n$ be another collection of open sets with the same properties.
Let $h_i\colon \partial U_i \to \partial U_i'$, $i\in \{1,\dots,N\}$, be orientation-preserving homeomorphisms.  Then there exists a homeomorphism
\[h\colon \S^n  \to \S^n   \] which extends $h_i$ for $i\in \{1,\dots,N\}$.
\end{prop}

\begin{proof}
We prove by induction on the number $N$ of open sets in each collection. The statement is true for $N=1$, by the flatness.

Suppose $N=2$. By the Annulus Theorem (Theorem \ref{Annulus:thm}), there exist homeomorphisms
\[\varphi\colon \S^n\setminus (U_1\cup U_2)\to \S^{n-1}\times[0,1]\quad \text{and}\quad \varphi'\colon \S^n\setminus (U_1'\cup U_2')\to \S^{n-1}\times[0,1].
 \]
Next, by the Isotopy Theorem (Theorem \ref{Isotopy:thm}), there exists a homeomorphism $F\colon \S^{n-1}\times[0,1]\to \S^{n-1}\times[0,1]$ which is an
isotopy between $\varphi'\circ h_0\circ \varphi^{-1}|_{\S^{n-1}\times \{0\}}$ and $\varphi'\circ h_1\circ \varphi^{-1}|_{\S^{n-1}\times \{1\}}$.
Then the homeomorphism  $h\coloneqq {\varphi'}^{-1}\circ F \circ \varphi \colon \S^n\setminus (U_1\cup U_2) \to \S^n\setminus (U_1'\cup U_2')$ extends
$h_1|_{\partial U_1}$ and $h_2|_{\partial U_2}$,  and $h$ may be extended homeomorphically to a map  $\S^n\to \S^n$ as claimed in the proposition.

Consider now the case when there are  $N$ open sets in each collection, and $N$ boundary homeomorphisms $h_i\colon \partial U_i \to \partial U_i'$, $i\in \{1,\dots,N\}$. Assume, by the induction hypothesis, that the proposition has been proved when the number is $N-1$. Fix now a homeomorphism  $g \colon \S^n\to \S^n$  which extends $h_i\colon \partial U_i \to \partial U_i'$ for $i\in \{1,\dots,N-1\}$.

We claim that there is a flat $n$-cell $D$ which contains $\br{U'_N}$ and $g(\br{U_N})$ in its interior and keeps $\bigcup_{i=1}^{N-1} \br{U'_i}$ in its complement.

To this end, one can first convert all cells $\br{U_i'}$, $i=1,\dots,N-1$, to round balls, by the induction assumption, with a global homeomorphism of $\S^n$. Thus, we suppose that $\br{U_i'}$, $i=1,\dots,N-1$, are round balls. Fix next $N-2$ pairwise disjoint flat $n$-cells, $L_1, \ldots, L_{N-2}$, contained in
 \[W\coloneqq \S^n\setminus \left( \br{U'_N} \cup g(\br{U_N}) \cup \bigcup_{i=1}^{N-1} U_i' \right), \]
which have the properties that (a) $L_i \cap \br{U_i'}$ and $L_i \cap \br{U_{i+1}'}$ are flat $(n-1)$-cells in $\partial U_i'$ and $\partial U_{j+1}'$, respectively, and $L_i\cap \br{U_j'}=\emptyset$ for $j\neq i, i+1$, (b)
 the set $E\coloneqq \br{U_1'}\cup L_1\cup\cdots\cup \br{U_{N-2}'}\cup L_{N-2}\cup \br{U_{N-1}'}$ is a flat $n$-cell, and  (c) the set $\br{\S^n\setminus E}$ is also a flat $n$-cell which contains $\br{U'_N} \cup g(\br{U_N})$ in its interior.
 Each of these $n$-cells $L_i$ can be constructed by connecting the balls $\br{U_i'}$ and $\br{U_{i+1}'}$ with a smooth path in $W$ and then fattening the path to obtain a smooth cylinder $L_i$.
By shrinking the flat cell $\br{\S^n\setminus E}$ we may obtain  a flat $n$-cell $D\subset \S^n\setminus E$ with the claimed properties.

Applying the initial step (for $N=2$) to the region $D\setminus  g(U_N)$, we obtain
a homeomorphism $f\colon D\setminus  g(U_N) \to D\setminus  U'_N$ which agrees with
the homeomorphisms $h_N \circ g^{-1}|_{\partial (g(U_N))}$ and $\text{id}|_{\partial D}$ on the boundary. The map $h=f \circ g\colon \S^n \setminus \bigcup_{i=1}^N  U_i \to \S^n \setminus \bigcup_{i=1}^N  U'_i $, satisfying $h|_{\partial U_i}=h_i$,  may then be extended to a homeomorphism $\S^n\to \S^n$.
\end{proof}

\begin{corollary}\label{Extension:cor}
Let $n\geq 3$,   $U_1,\dots,U_N\subset \S^n$ be open sets having pairwise disjoint closures and for which $\S^n\setminus U_1,\dots,\S^n\setminus U_N$ are $n$-cells, and let $U'_1,\dots,U'_N\subset\S^n$ be another collection of open sets with the same properties.
Let $h_i\colon \partial U_i \to \partial U_i'$, $i\in \{1,\dots,N\}$, be orientation-preserving homeomorphisms.  Then there exists a homeomorphism \[h\colon \S^n \setminus \bigcup_{i=1}^N  U_i \to \S^n \setminus \bigcup_{i=1}^N  U'_i  \] which extends $h_i$ for $i\in \{1,\dots,N\}$.
\end{corollary}

Corollary \ref{Extension:cor} follows from Proposition \ref{Extension:prop} as follows.  Following the proof of Lemma 0 in \cite{Cannon:Sierpinski}, we can ``enlarge"  the complementary components $U_i$, $i\in \{1,\dots,N\}$,
slightly by an embedding
$\psi\colon  \S^n\setminus \bigcup_{i=1}^N  U_i \hookrightarrow \S^n\setminus \bigcup_{i=1}^N  U_i $ in such a way that the boundary components $\psi(\partial U_i)$ of the new regions are flat $(n-1)$-spheres. We do the same for $\S^n\setminus \bigcup_{i=1}^N  U'_i $ by an embedding $\psi'$. We apply Proposition \ref{Extension:prop} to the new regions to find a homeomorphic extension, and  then use $\psi^{-1}$ and $ \psi'^{-1}$ to pull the extension back to the original regions.
\medskip

We also need the following approximation theorem for cell-like and cellular maps; recall that cellular maps are cell-like.

\begin{theorem}[Approximation Theorem for Cell-like/Cellular Maps]\label{Approximation:thm}
Let $M,N$ be $n$-manifolds without boundary and let $f\colon M \to N$ be a cell-like map when $n\geq 4$, or a  cellular map when  $n=3$. Suppose that $\rho$ is a metric on $N$, $C$ is a closed subset of $N$ for which $f|_{f^{-1}(C)}$ is injective, and $\epsilon \colon N\to [0,\infty)$ is a continuous function satisfying $\epsilon^{-1}(0)=C$. Then there is a homeomorphism $g\colon M\to N$ satisfying
$\rho(g(x),f(x))<\epsilon(f(x))$ for all $x\in M\setminus f^{-1}(C)$ and $g|_{f^{-1}(C)}=f|_{f^{-1}(C)}$.
\end{theorem}

Theorem \ref{Approximation:thm} was proved for dimension $3$ in \cite{Armentrout} and, for $3$-manifolds with boundary, in \cite{Armentrout_with_boundary}; for  dimension $n\geq 5$ and, by a different method, for $n=3$ in \cite{Siebenmann}; and for dimension $4$ in \cite{Ancel}. The statement above  is adapted from the cited works above,  and also
\cite[Corollary 7.4.3]{Daverman-Venema}. For our application, we need the following corollary.

\begin{corollary}\label{Approximation:cor}
Let $f\colon M\to N$ be a cellular map between $n$-manifolds without boundary, where $n\geq 3$. Let $B\subset N$ be an $n$-cell and $C=\partial B$.  Suppose that $f$ is injective on $f^{-1}(C)$. Then $f^{-1}(B)$ is an $n$-cell in $M$, whose boundary is $f^{-1}(C)$.
\end{corollary}

\bigskip

\bibliography{biblio-Sierpinski}

\end{document}